\documentclass[pdflatex,sn-mathphys-num]{sn-jnl}

\usepackage{graphicx}%
\usepackage{multirow}%
\usepackage{amsmath,amssymb,amsfonts}%
\usepackage{amsthm}%
\usepackage{mathrsfs}%
\usepackage[title]{appendix}%
\usepackage{xcolor}%
\usepackage{textcomp}%
\usepackage{manyfoot}%
\usepackage{booktabs}%
\usepackage{algorithm}%
\usepackage{algorithmicx}%
\usepackage{algpseudocode}%
\usepackage{listings}%
\usepackage{tikz}
\usepackage{diagbox}

\theoremstyle{thmstyleone}%
\newtheorem{theorem}{Theorem}
\theoremstyle{thmstyletwo}%
\newtheorem{lemma}{Lemma}%
\newtheorem{remark}{Remark}%

\theoremstyle{thmstylethree}%

\begin{document}

\title[TSD: A Geometry-Based Gradient Algorithm with Guaranteed R-Linear Convergence]{Triangle Steepest Descent: A Geometry-Based Gradient Algorithm with Guaranteed R-Linear Convergence}


\author[1]{\fnm{Ya} \sur{Shen}}\email{ya.shen200012@gmail.com}

\author*[1,2]{\fnm{Qing-Na} \sur{Li}}\email{qnl@bit.edu.cn}
\equalcont{These authors contributed equally to this work.}

\author[3,4]{\fnm{Yu-Hong} \sur{Dai}}\email{dyh@lsec.cc.ac.cn}
\equalcont{
This work was supported by National Natural Science Foundation of China grants (No. 11271032 and No. 12271526).}

\affil*[1]{\orgdiv{School of Mathematics and Statistics}, \orgname{Beijing Institute of Technology}, \orgaddress{
\city{Beijing}, \postcode{100081}, \state{Beijing}, \country{China}}}

\affil[2]{\orgdiv{Beijing Key Laboratory on MCAACI, Key Laboratory of Mathematical Theory and Computation in Information Security},  \orgname{Beijing Institute of Technology}, \orgaddress{
\city{Beijing}, \postcode{100081}, \state{Beijing}, \country{China}}}


\affil[3]{\orgdiv{LSEC, ICMSEC, AMSS}, \orgname{Chinese Academy of Sciences}, \orgaddress{
\city{Beijing}, \postcode{100190}, \state{Beijing}, \country{China}}}

\affil[4]{\orgdiv{School of Mathematical Sciences}, \orgname{University of Chinese Academy of Sciences}, \orgaddress{
\city{Beijing}, \postcode{100049}, \state{Beijing}, \country{China}}}


\abstract{
Gradient methods are among the simplest yet most widely used algorithms for unconstrained optimization. Motivated by a geometric property of the steepest descent (SD) method that can alleviate the zigzag behavior in quadratic problems, we develop a new gradient variant called the Triangle Steepest Descent (TSD) method. The TSD method introduces a cycle parameter $j$ that governs the periodic combination of past search directions, providing a geometry-driven mechanism to enhance convergence. To the best of our knowledge, TSD is the first formally established geometry-based gradient scheme since Akaike (1959). We prove that TSD is at least R-linearly convergent for strongly convex quadratic problems and demonstrate through extensive numerical experiments that it exhibits superlinear behavior, outperforming the Barzilai-Borwein (BB) method and monotone Dai-Yuan gradient method (DY) in quadratic cases. These results suggest that incorporating geometric information into gradient directions offers a promising avenue for developing efficient optimization algorithms.
}




\maketitle
\section{Introduction}
\label{intro}
Gradient methods for solving the optimization problem 
\begin{equation}
\label{object_function_intro}
\min_{x\in\mathbb{R}^n} f(x)
\end{equation} 
where $f(x)$ is continuously differentiable, take the form of 
\begin{equation}
x_{k+1}=x_k-\alpha_kg_k,
\end{equation}
 where $g_k$ is the gradient of $f(x)$ at point $x_k$ and $\alpha_k$ is the stepsize. The well-known steepest descent method (SD) proposed by Cauchy calculates $\alpha_k$ by the exact line search. 
However, it presents zigzag behavior and slow Q-linear convergence in the quadratic situation \cite{Akaike1959}. 

In 1988, Barzilai and Borwein proposed the BB method \cite{Barzilai1988} by exploiting the quasi-Newton property. Barzilai and Borwein proved that in a 2-dimensional special  case, the BB method is R-superlinearly convergent \cite{Barzilai1988}. For strongly convex quadratic case in general dimension, it is globally and R-linearly convergent \cite{Dai2002}. It is also adapted to solve non-quadratic functions \cite{Dai2006,Raydan1997} and constrained optimization problems \cite{Birgin2000,Dai2005b,Dai2006a,Hager2006,Serafini2005}. Compared with other gradient methods, the BB method has less computational complexity because there is no matrix-vector product involved. Furthermore, as pointed out by Fletcher \cite{Fletcher2005}, the BB method is more stable for general optimization problems than the conjugate gradient method, while the conjugate gradient method only works well for quadratic models.

In 2002, Raydan and Svaiter \cite{Raydan2002} proposed the relaxed Cauchy method (RC, also named relaxed steepest descent method) and the Cauchy-Barzilai-Borwein method (CaBB). CaBB method and the alternate step (AS) gradient method \cite{Dai2003} are the same method proposed independently, both showing that, with the convex quadratic setting,
\begin{equation}
\label{fx}
f(x)=\frac12x^{\top}Ax+b^{\top}x,
\end{equation}
where $A\in\mathbb{R}^{n\times n}$ is symmetric positive definite, $b\in\mathbb{R}^n$, 
 CaBB can produce gradients approximating eigenvector of $A$ and stepsizes approximating eigenvalue of $A$. The RC method can improve the SD method greatly, and the CaBB method can improve the BB method greatly \cite{Raydan2002}. Based on this point, a natural thought is to extend the CaBB method by using more repeated stepsizes to obtain better performance.

Actually, this clue of reusing previous stepsizes to accelerate gradient methods is termed as the cyclic strategy, which work very well in practice \cite{Dai2005,Frassoldati2008,Asmundis2014,Zou2018,OviedoLeon2021,Huang2022}. It is first mentioned in 1999, Friedlander et al. \cite{Friedlander1998} introduced a class of gradient methods with retards and pointed out that this class of gradient methods with retards including the cyclic SD method (CSD) can get better performance. In 2005, Dai and Fletcher \cite{Dai2005a} analyzed  theoretical properties of the CSD method, concluding that the CSD method may not work well if the cycle number $m$ is not big. In 2006, Dai et al. \cite{Dai2006} proposed the cyclic Barzilai-Borwein method (CBB) and the adaptive cyclic Barzilai-Borwein method (ACBB). They showed that these methods improve the BB method and ACBB is comparable to or even better than conjugate gradient algorithms if the objective function is highly non-linear. Afterwards, the cyclic strategy is widely used \cite{Dai2005,Frassoldati2008,Asmundis2014,Zou2018,OviedoLeon2021,Huang2022}.

Another strategy for accelerating gradient methods is by equipping a gradient method with the two-dimensional quadratic property. Yuan stepsize \cite{Yuan2006,Yuan2008} can achieve the minimizer of a two-dimensional strictly convex quadratic function in three iterations, by cooperating with the SD stepsize in an alternate manner. Dai-Yuan method \cite{Dai2005} is a generalization of the former method, it can outperform the nonmonotone BB method. The BBQ stepsize proposed by Huang \cite{Huang2021}, computed from a necessary condition ensuring the next gradient parallels to some eigenvector of the Hessian, can not only ensures two-dimensional quadratic termination property, but also is extended for unconstrained optimization with the Grippo-Lampariello-Lucidi (GLL) nonmonotone line search (This line search is much cheaper than the exact line search).

Of course, there are many other strategies for accelerating gradient methods, but we do not intend to elaborate further on them. In summary, the majority of existing acceleration techniques for gradient methods focus on refining the stepsize, either through quasi-Newton approximations (BB-type methods), cyclic reuse, or two-dimensional quadratic termination properties. However, despite these advances, the idea of systematically exploiting search directions has not been formalized since Akaike (1959)  \cite{Akaike1959}. This gap provides the main motivation for our work.

Unlike existing gradient variants that focus solely on stepsize tuning, the triangle steepest descent method (TSD) proposed by us is the first method in 66 years to exploit the potential of search directions. By periodically aggregating past gradient directions, TSD inherently avoids zigzagging while maintaining simplicity. Although Akaike (1959) \cite{Akaike1959} hinted at a similar idea, our work establishes its at least R-linearly convergence theory on strongly convex case and conducts extensive numerical experiments on it. Numerical results show that TSD has superlinear convergence behavior for strongly convex quadratic problems, and outperforms the BB and DY methods in ill-conditioned cases, although its performance is worse than BBQ and ABBmin2, and its directional strategy opens a new avenue for gradient methods. The detail of the idea for TSD is in next section. 

In this paper, our main contributions are as follows:
(i), we propose TSD, (ii), We establish its at least R-linear convergence for strongly convex quadratic problems, which, to the best of our knowledge, is the first convergence guarantee for a method exploiting this specific geometric idea, (iii), we validate its performance numerically. 

 The organization of this paper is as follows. 
In Section \ref{TSD}, we introduce our main idea of TSD. 
In Section \ref{sec_convergence}, we show its R-linear convergence rate for strongly convex quadratic problems. 
In Section \ref{results}, extensive numerical results confirm our theoretical results and show the effectiveness of TSD. We draw our conclusions in Section \ref{conclusion}.

In this paper, we use subscript $k$ to denote the k-th iteration.
We use $Diag(x)$ to denote the diagonal matrix with diagonal elements in a vector $x\in\mathbb{R}^n$.

\section{The TSD method}

\label{TSD}
In this section, we will introduce the triangle steepest descent method (TSD). 
%
Without loss of generality, we consider the strongly convex quadratic problem as defined in {\rm(\ref{fx})}, where 
\begin{equation}
\label{fa}
A=Diag\left(\lambda_1, \cdots, \lambda_n\right),\ 0<\lambda_1<\cdots<\lambda_n.
\end{equation}
\begin{figure}[htbp]
\centering
\begin{tikzpicture}[scale=3]
\filldraw [black] (0cm,0cm) circle [radius=0.15mm];
\draw (0cm,0.1cm) node[] {O};
\draw (0,0) ellipse [x radius=2cm, y radius=1cm];
\draw (0,0) ellipse [x radius=1cm, y radius=0.5cm];
  \draw (-1.75cm,-0.4841cm) -- (0,0);
\filldraw [black] (-1.75cm,-0.4841cm) circle [radius=0.15mm];
\draw (-1.75cm,-0.3841cm) node[] {\large$x_0$};
 \draw (-1.75cm,-0.4831cm) --(-1.0899479cm,0.2462571cm) ;
\draw (0,0) ellipse [x radius=1.196059cm, y radius=0.5980294cm];
\filldraw [black] (-1.0899479cm,0.2462571cm) circle [radius=0.15mm];
\draw (-1.0899479cm,0.3462571cm) node[] {\large$x_1$};
\draw (-1.0899479cm,0.2462571cm) -- (-0.6258824cm,-0.1731370cm);
\draw (0,0) ellipse [x radius=0.7152862cm, y radius=0.3576431cm];
\filldraw [black] (-0.6258824cm,-0.1731370cm) circle [radius=0.15mm];
\draw (-0.6258824cm,-0.0731370cm)  node[] {\large$x_2$};
\draw[black, very thick, ->, >=latex] (-1.75cm,-0.4841cm) -- (-0.6258824cm,-0.1731370cm);
\node[black] at (-0.5cm,-0.3cm) {\fontsize{9}{12}\selectfont Direction to minimum};
 \draw  (-0.6258824cm,-0.1731370cm) -- (-0.38981670cm,0.08807313cm);
\draw (0,0) ellipse [x radius=0.4277669cm, y radius= 0.2138835cm];
\filldraw [black] (-0.38981670cm,0.08807313cm) circle [radius=0.15mm];
\draw  (-0.38981670cm,0.18807313cm)node[] {\large$x_3$};
  \draw  (-0.38981670cm,0.08807313cm) -- (-0.22384502cm,-0.06192193cm);
\draw (0,0) ellipse [x radius=0.25582cm, y radius=  0.12791cm];
\filldraw [black] (-0.22384502cm,-0.06192193cm) circle [radius=0.15mm];
\draw  (-0.22384502cm,0.04cm) node {\large$x_4$};
  \draw  (-0.22384502cm,-0.06192193cm) -- (-0.1394168cm,0.0314991cm);
\draw (0,0) ellipse [x radius=0.1529896cm, y radius=0.0764948cm];
\filldraw [black] (-0.1394168cm,0.0314991cm) circle [radius=0.15mm];
  \draw   (-0.1394168cm,0.0314991cm) -- (-0.08005752cm,-0.02214620cm);
\draw (0,0) ellipse [x radius=0.09149329cm, y radius=0.04574665cm];
\filldraw [black] (-0.08005752cm,-0.02214620cm) circle [radius=0.15mm];
\end{tikzpicture}
\caption{General situation diagram of 2-dimensional SD method}
\label{2-dimensionalSD}       
\end{figure}

To derive the TSD method, we start with a property of the SD method proved by Akaike \cite{Akaike1959}, that is, for a strongly convex quadratic problem, the directions of gradients generated from SD tend to two different directions:
\begin{equation}
d_1=\lim_{k\rightarrow\infty}\frac{g_{2k}}{\|g_{2k}\|_2}, \quad d_2=\lim_{k\rightarrow\infty}\frac{g_{2k-1}}{\|g_{2k-1}\|_2}.
\end{equation}

%
Furthermore, Theorem 2.4 in \cite{Huang2022} shows that this property actually can be extended to all gradient methods whose stepsizes are reciprocals of some of Rayleigh quotients of $A$. 

Figure \ref{2-dimensionalSD} is an illustration of the first few iterations of the SD method in a 2-dimensional situation. In 2-dimensional situations, since $g_k^Tg_{k+1}=0$, equally, $g_k\perp g_{k+1}$, so, $g_k\parallel g_{k+2}$, for all $k\in\mathbb{Z}_+$,  for SD, thus the triangle determined by $x_0$, $x_1$, $x_2$ is similar to the triangle determined by $x_2$, $x_3$, $x_4$. Thus, we know, in the case here, the SD method will linearly converge to (0, 0) and will never equal to (0, 0) in any iteration.

Combining the above property of SD and the fact that Figure \ref{2-dimensionalSD} shows that the direction from $x_0$ to $x_2$ points toward the minimum, we propose the triangle steepest descent method (TSD). 
See Algorithm \ref{TSDal} for the details of TSD. In Algorithm \ref{TSDal}, $j\ge3$ is a parameter. The stepsize $\alpha_k^{TSD}$ is given by the exact line search

\begin{equation}
\label{TSDeq}
\alpha_k^{TSD}=-\frac{p_k^{\top}g_k}{p_k^{\top}Ap_k},
\end{equation}
where 
\begin{equation}
p_k=\left\{\begin{aligned}&-\alpha_{k-1}^{TSD}g_{k-1}-\alpha_{k-2}^{TSD}g_{k-2}, &k\ne0\ and\mod(k,j)=0,\\
 &-g_k, &\hbox{otherwise}.
\end{aligned}\right.
\label{pk}
\end{equation}
\begin{algorithm}[h]
	\caption{The TSD method for solving (\ref{fx})}
    \label{TSDal}
    \begin{algorithmic}[0] 
 \renewcommand{\algorithmicrequire}{\textbf{Input:}}
        \Require  initial point $x_0$, integer $j\ge 3$, tolerance $\epsilon>0$, $k:=0$. 
\State S1: If $\|g_k\|\le \epsilon$, stop, otherwise, compute $p_k$ by (\ref{pk});
\State S2: Compute $\alpha_k^{TSD}$ by (\ref{TSDeq});
\State S3: $x_{k+1}=x_k+\alpha_k^{TSD}p_k$. $k:=k+1$. Go to S1.
    \end{algorithmic}
\end{algorithm}


\section{The convergence rate of TSD}
\label{sec_convergence}
First, we have the following lemma about TSD.
\begin{lemma}
\label{TSDlem}
For the TSD method in Algorithm \ref{TSDal} to solve {\rm(\ref{object_function_intro})} with $f(x)$ defined in {\rm(\ref{fx})} and {\rm(\ref{fa})}, $p_k$ defined by (\ref{pk}) is a descent direction.
\end{lemma}
\begin{proof}
$\ $
Since $-g_k$ is a descent direction, it suffices to prove $p_k$ is descent when $k\ne0$ and $mod(k,j)=0$.
When $k\ne0$ and $mod(k,j)=0$, 
$$g_{k}=Ax_{k}+b=A\left(x_{k-1}-\alpha_{k-1}^{T S D} g_{k-1}\right)+b=\left(I-\alpha_{k-1}^{T S D}A\right) g_{k-1}.$$
Similarly, $g_{k-1}=\left(I-\alpha_{k-2}^{T S D} A\right) g_{k-2}$. As a result, it holds that
$$g_{k}^{\top} g_{k-1}=g_{k-1}^{\top}\left(I-\alpha_{k-1}^{T S D} A\right) g_{k-1}=g_{k-1}^{\top} g_{k-1}-\frac{g_{k-1}^{\top} g_{k-1}}{g_{k-1}^{\top} A g_{k-1}} g_{k-1}^{\top} A g_{k-1}=0.$$
Similarly,  $g_{k-1}^{\top} g_{k-2}=0$.

It holds that
$$
\begin{aligned}
g_{k}^{\top} g_{k-2}&=g_{k-1}^{\top}\left(I-\alpha_{k-1}^{T S D} A\right) g_{k-2}\\
&=g_{k-1}^{\top} g_{k-2}-\alpha_{k-1}^{T S D} g_{k-1}^{\top} A g_{k-2} \\
&=-\alpha_{k-1}^{T S D} g_{k-1}^{\top} A g_{k-2}\\
&=-\alpha_{k-1}^{T S D} g_{k-2}^{\top}\left(I-\alpha_{k-2}^{T S D} A\right) A g_{k-2}.
\end{aligned}$$
By Cauchy-Schwartz inequality, it holds that
$$g_{k-2}^{\top}\left(I-\alpha_{k-2}^{T S D} A\right) A g_{k-2}=g_{k-2}^{\top} A g_{k-2}-\frac{g_{k-2}^{\top} g_{k-2} g_{k-2}^{\top} A^{2} g_{k-2}}{g_{k-2}^{\top} A g_{k-2}} \leq 0.$$
The equal sign holds if and only if there exists $\hat{\lambda}\in\mathbb{R}$, such that  $A g_{k-2}=\hat{\lambda} g_{k-2}$. But, if so, then 
$\alpha_{k-2}^{T S D}=\frac{g_{k-2}^{\top} g_{k-2}}{g_{k-2}^{\top} A g_{k-2}}=\frac{1}{\hat{\lambda}}$,
$g_{k-1}=\left(I-\alpha_{k-2}^{T S D} A\right) g_{k-2}=0$. This is a contradiction with $g_{k-1}\ne 0$, since we are computing $p_k$ and $\alpha_k^{TSD}$ now. So, the equal sign does not hold. 

On the other hand, 
$$\alpha_{k-2}^{TSD}=\frac{g_{k-2}^{\top}g_{k-2}}{g_{k-2}^{\top}Ag_{k-2}}>0,\ \alpha_{k-1}^{TSD}=\frac{g_{k-1}^{\top}g_{k-1}}{g_{k-1}^{\top}Ag_{k-1}}>0.$$ Therefore,
$$g_{k}^{\top} g_{k-2}>0.$$
So, we get that $p_k^{\top}g_k=g_{k}^{\top}\left(-\alpha_{k-1}^{T S D} g_{k-1}-\alpha_{k-2}^{T S D} g_{k-2}\right)=-\alpha_{k-2}^{T S D}g_{k}^{\top} g_{k-2}<0$.
The proof is finished.
\end{proof}

In geometric perspective, Lemma \ref{TSDlem} is shown as the fact that the non-right angle in a right-angled triangle must be acute. Such as $\angle x_2x_0x_1$ in Fig. \ref{2-dimensionalSD}.

%
In the following, we can prove the R-linear convergence result of TSD.

\begin{theorem}
\label{SD2_con}
For the TSD method to solve {\rm(\ref{object_function_intro})} with $f(x)$ defined in {\rm(\ref{fx})} and {\rm(\ref{fa})}, the sequence $\left\{x_k\right\}$ converge at least R-linearly to the minimum with converging factor $\left(\frac{\kappa-1}{\kappa+1}\right)^{\frac{j-1}{j}}$, where $\kappa$ is the condition number of $A$ defined by $\kappa=\frac{\lambda_n}{\lambda_1}$. 
\end{theorem}
\begin{proof}
We need to show $\lim\limits_{k\rightarrow\infty}\|x_k-x^*\|\le c\left(\left(\frac{\kappa-1}{\kappa+1}\right)^{\frac{j-1}{j}}\right)^k$, where $c$ is a positive constant, $x^*$ is the minimum point. 
Here we use the norm $\|x\|_A\triangleq\langle x,Ax\rangle$. 
%

When $k \neq n j$ or $k = 0$, by (\ref{TSDeq}) and (\ref{pk}), it holds that $p_{k}=-g_{k}$, $\alpha_{k}^{TSD}=\alpha_{k}^{SD}=\frac{g_{k}^{\top} g_{k}}{g_{k}^{\top} A g_{k}}$. Together with $g_{k}=A\left(x_{k}-x^{*}\right)$, we obtain that
\begin{equation}
\begin{aligned}
\label{iter}
\frac{\left\|x_{k+1}-x^{*}\right\|_{A}^{2}}{\left\|x_{k}-x^{*}\right\|_{A}^{2}}
&=\frac{\left(x_{k}-\alpha_{k}^{T S D} g_{k}-x^{*}\right)^{\top} A\left(x_{k}-\alpha_{k}^{TSD} g_{k}-x^{*}\right)}{\left(x_{k}-x^{*}\right)^{\top} A\left(x_{k}-x^{*}\right)}\\
&=1-\frac{\left(g_{k}^{\top} g_{k}\right)^{2}}{\left(g_{k}^{\top} A g_{k}\right)\left(g_{k}^{\top} A^{-1} g_{k}\right)}.
\end{aligned}
\end{equation}
Using Kantorovich inequality and notice that $A$ is a positive definite symmetric matrix, when $g_{k} \neq 0$, it holds that
\begin{equation}
\label{Kantorovich}
\frac{\left(g_{k}^{\top} g_{k}\right)^{2}}{\left(g_{k}^{\top} A g_{k}\right)\left(g_{k}^{\top} A^{-1} g_{k}\right)} \geqslant \frac{4 \lambda_n \lambda_1}{\left(\lambda_n+\lambda_1\right)^{2}},
\end{equation}
where $\lambda_1$ and $\lambda_n$ are the smallest and the largest eigenvalue of $A$ respectively. 

Combining (\ref{iter}) and (\ref{Kantorovich}), we get
$$\frac{\left\|x_{k+1}-x^{*}\right\|_{A}^{2}}{\left\|x_{k}-x^{*}\right\|_{A}^{2}} \leqslant\left(\frac{\lambda_n-\lambda_1}{\lambda_n+\lambda_1}\right)^{2},$$
that is,
$$\frac{\left\|x_{k+1}-x^{*}\right\|_{A}}{\left\|x_{k}-x^{*}\right\|_{A}} \leqslant \frac{\kappa-1}{\kappa+1},\quad \kappa=\frac{\lambda_n}{\lambda_1}.$$

When $k\ne0$ and $mod(k,j)=0$, 
by Lemma \ref{TSDlem}, it holds that $f(x_{k+1})<f(x_k)$. In fact, we know the minimal point of our object problem is $x^*=-A^{-1}b$, so, equivalently it holds that $\frac12(x_{k+1}-x^*)^{\top}A(x_{k+1}-x^*)-\frac12{x^*}^{\top}Ax^*<\frac12(x_{k}-x^*)^{\top}A(x_{k}-x^*)-\frac12{x^*}^{\top}Ax^*$, equivalently that is $\frac{\|x_{k+1}-x^*\|_A}{\|x_{k}-x^*\|_A}<1$.

So, for the TSD method, it holds that 
$$\left\{\begin{aligned}
&\frac{\|x_{k+1}-x^*\|_A}{\|x_{k}-x^*\|_A}\le\frac{\kappa-1}{\kappa+1}, &when\ mod(k,j)\ne0,\\
&\frac{\|x_{k+1}-x^*\|_A}{\|x_{k}-x^*\|_A}<1, &\hbox{otherwise}.
\end{aligned}\right.$$

Denote $a_0=\|x_0-x^*\|$, $\forall s\in\mathbb{Z}^+$, it holds that
$$\left\{\begin{aligned}
&\|x_{sj}-x^*\|_A=a_0\prod_{k=0}^{sj}\frac{\|x_{k+1}-x^*\|_A}{\|x_{k}-x^*\|_A}< a_0\left(\frac{\kappa-1}{\kappa+1}\right)^{s(j-1)}, &\\
&\|x_{sj+p}-x^*\|_A=a_0\prod_{k=0}^{sj+p}\frac{\|x_{k+1}-x^*\|_A}{\|x_{k}-x^*\|_A}< a_0\left(\frac{\kappa-1}{\kappa+1}\right)^{s(j-1)+p-1}, 
&p=2,\cdots, j-1,\\
&\|x_{sj+1}-x^*\|_A=a_0\prod_{k=0}^{sj+1}\frac{\|x_{k+1}-x^*\|_A}{\|x_{k}-x^*\|_A}< a_0\left(\frac{\kappa-1}{\kappa+1}\right)^{s(j-1)}. &
\end{aligned}\right.$$

That is,
$$\left\{\begin{aligned}
&\|x_{k}-x^*\|_A=a_0\prod_{i=0}^{k}\frac{\|x_{i+1}-x^*\|_A}{\|x_{i}-x^*\|_A}\le a_0\left(\frac{\kappa-1}{\kappa+1}\right)^{k\frac{j-1}{j}}, &k=sj\\
&\|x_{k}-x^*\|_A=a_0\prod_{i=0}^{k}\frac{\|x_{i+1}-x^*\|_A}{\|x_{i}-x^*\|_A}\le a_0\left(\frac{\kappa-1}{\kappa+1}\right)^{k\frac{sj-s+p-1}{sj+p}}, &k=sj+p,\quad\ 
p=2,\cdots, j-1,\\
&\|x_{k}-x^*\|_A=a_0\prod_{i=0}^{k}\frac{\|x_{i+1}-x^*\|_A}{\|x_{i}-x^*\|_A}\le a_0\left(\frac{\kappa-1}{\kappa+1}\right)^{k\frac{sj-s}{sj+1}}, &k=sj+1.
\end{aligned}\right.$$
Thus, we get for the TSD method 
$$\lim\limits_{k\rightarrow\infty}\|x_{k}-x^*\|_A\le \lim\limits_{s\rightarrow\infty}a_0\left(\frac{\kappa-1}{\kappa+1}\right)^{k*\min\left\{\frac{j-1}{j}, \frac{sj-s+p-1}{sj+p}, \frac{sj-s}{sj+1}\right\}}=a_0\left(\left(\frac{\kappa-1}{\kappa+1}\right)^{\frac{j-1}{j}}\right)^k.$$

This result shows that the iteration sequence $x_k$ at least converge R-linearly with a factor $\left(\frac{\kappa-1}{\kappa+1}\right)^{\frac{j-1}{j}}$.
\end{proof}
%

\begin{remark}
Theorem \ref{SD2_con} shows that TSD is R-linear convergent for strongly convex quadratic problems, but it does not reveal that TSD is superior to SD. In fact, we are hoping that $\frac{\|x_{sj+1}-x^*\|_A}{\|x_{sj}-x^*\|_A}\ll1$ in TSD for positive integers $s$, $j$ to guarantee TSD's good effect in practice. $\frac{\|x_{sj+1}-x^*\|_A}{\|x_{sj}-x^*\|_A}\ll1$ often occurs when j is not so small. Specifically,  with the results of Theorem 4 in \cite{Akaike1959}, for a specific quadratic problem, for all $\epsilon\in(0,1)$, there exists $J$ is a positive integer, such that $\frac{\|x_{sj+1}-x^*\|_A}{\|x_{sj}-x^*\|_A}<\epsilon$, for all $j>J$. However, for all general strongly convex quadratic problems, it's hard to say whether there is such a $J$ can guarantee the same result.
\end{remark}

\section{TSD algorithm for unconstrained optimization}
 In this section, we discuss unconstrained optimization algorithms related to the TSD method. A gradient algorithms will be proposed that are mainly based on the TSD method.

We first introduce the  strong Wolfe line search condition for computing the stepsize $\alpha_k$  \cite{JorgeNocedal2006}: 
\begin{align}
f(x_k)-f(x_k-\alpha_kp_k)&\ge c_1\alpha_kg_k^{\top}p_k,\label{W1}\\
|g(x_k-\alpha_kp_k)^{\top}p_k|&\le c_2|g_k^{\top}p_k|, \label{W2}
\end{align}
where $0<c_1<c_2<1$, stepsize $\alpha_k>0$, $p_k$ is a descent direction. Follow \cite{Dai2003}, in our numerical experiment, we will use $c_1=10^{-4}$, $c_2=0.1$, and the line search begins from a unit initial stepsize. 

In this way, the TSD algorithm for unconstrained optimization is given as follows. 
\begin{algorithm}[!ht]
	\caption{The TSD algorithm for unconstrained optimization}
    \label{TSDaluo}
    \begin{algorithmic}[0] 
 \renewcommand{\algorithmicrequire}{\textbf{Input:}}
        \Require  initial point $x_0\in\mathbb{R}^n$, $c_1=10^{-4}$, $c_2=0.1$, integer $j\ge 3$, tolerance $\epsilon>0$, $k:=0$. 
\State S1: If $\|g_k\|\le \epsilon$, stop, otherwise, compute $p_k$ by (\ref{pk});
\State S2: Compute $\alpha_k^{TSD}$ by the  strong Wolfe line search (\ref{W1}), (\ref{W2});
\State S3: $x_{k+1}=x_k+\alpha_k^{TSD}p_k$. $k:=k+1$. Go to S1.
    \end{algorithmic}
\end{algorithm}

Suppose that  $f$  is bounded below in  $\mathbb{R}^{n}$  and that  $f$ is continuously differentiable in an open set  $\mathcal{N}$  containing the level set  $\mathcal{L} \stackrel{\text { def }}{=}\left\{x: f(x) \leq f\left(x_{0}\right)\right\}$, where  $x_{0}$  is the starting point of the iteration. Assume also that the gradient  $g=\nabla f$ is Lipschitz continuous on $\mathcal{N}$. So, as a corollary of Theorem 3.2. in \cite{JorgeNocedal2006}, it holds that Algorithm \ref{TSDaluo} with $\epsilon=0$ converges in the sense that the gradient norms $\|g_k\|$ converge to $0$. 

\section{Numerical results}
\label{results}
In this section, we evaluate TSD on strongly convex quadratic problems and general unconstrained problems from the CUTEst collection \cite{Gould2014}. This section consists of three parts: (i), we investigate the performance of the TSD method on different values of $j$ and different condition numbers in quadratic problems. (ii), we compare the TSD method with other competitive gradient methods in quadratic problems.  (iii), we explore the performance of the TSD method and the BBQ method \cite{Huang2021} on general unconstrained problems from the CUTEst collection \cite{Gould2014} with dimension less than or equal to 10000. All the compared methods were implemented on a PC with an AMD Ryzen 5 4500U, 2.38 GHz processor and 16 GB of RAM running Windows 11 system. Matlab (v.9.10.0-R2021a) is used for the first 2 parts, Julia (version 1.11.5) is used for the third part. Main part of our codes is available online (see Declarations section).

\subsection{Changes in $j$ as well as condition number}
\label{4_1}

We will take a look at the performance of the TSD method due to the effect of parameter $j$ and the condition number. 

As to the range of $j$. Based on Theorem \ref{SD2_con}. The bigger the $j$ is, the bigger the index $\frac{j-1}{j}$, the bigger the R-linear convergence rate. However, when $j=100$, the index $\frac{j-1}{j}=0.99$ is close to its upper bound 1 enough. Based on Theorem 4 in \cite{Akaike1959}, the bigger the $j$ is, the higher possibility that the gradients will be alternate in two directions, which means better performance of TSD. Therefore, we will test some big values of $j$, we decide to set $j\in\{10,50,100,200,300,500,1000\}$.

We tested the following quadratic problem \cite{Yuan2006}:

\begin{equation}
\min _{x \in \mathbb{R}^{n}} f(x)=\left(x-x^{*}\right)^{\top} \operatorname{diag}\left\{v_{1}, \ldots, v_{n}\right\}\left(x-x^{*}\right),
\end{equation}
where  $x^{*}$  was randomly generated with components in  $[-10,10]$  and  $v_{i}$, $i=1,\ 2,\ \ldots,\ n $, were generated according to five different distributions listed as follows.
%

\begin{itemize}
\item \textbf{set 1} $\left\{v_{2}, \ldots, v_{n-1}\right\} \subset(1, \kappa), \quad v_{1}=1, \quad v_{n}=\kappa$.
\item \textbf{set 2} $v_{i}=1+(\kappa-1) s_{i}, \quad s_{i} \in(0.8,1), \quad i=1, \ldots, \frac{n}{2} $, \quad
  $s_{i} \in(0,0.2), \quad i=\frac{n}{2}+1, \ldots, n$.
\item \textbf{set 3} $\left\{v_{2}, \ldots, v_{n / 5}\right\} \subset(1,100), \quad\left\{v_{n / 5+1}, \ldots, v_{n-1}\right\} \subset\left(\frac{\kappa}{2}, \kappa\right), \quad v_{1}=1, \quad v_{n}=\kappa$.
\item \textbf{set 4} $v_{i}=\kappa^{\frac{n-i}{n-1}}$.
\item \textbf{set 5}  $\left\{v_{2}, \ldots, v_{4 n / 5}\right\} \subset(1,100), \quad\left\{v_{4 n / 5+1}, \ldots, v_{n-1}\right\} \subset\left(\frac{\kappa}{2}, \kappa\right), \quad v_{1}=1, \quad v_{n}=\kappa$.
\end{itemize}

The problem dimension was set to  $n=10000$ in this test. The iteration was stopped once the gradient norm reduces by a factor of  $\epsilon$  or exceeds the maximum iteration number $1e5$. We test $\kappa\in\{  10^{4},10^{5}, 10^{6}, 10^{10}\}$ and  $\epsilon\in\{10^{-6}, 10^{-9}, 10^{-12}\}$. For each value of $\kappa$ or $\epsilon$, average number of iterations over 10 different starting points with entries randomly generated in $[-10,10]$ are presented in Table \ref{result3}.

From the result here, we can see that the iteration number of TSD increases as $\kappa$ increases. This is consistent with our intuition. TSD with $j=10$ performs best when $\kappa=1e4$ for all problem set. But when $\kappa=1e5$, $1e6$, $1e10$, TSD with $j=50$, $100$ perform best for problem set 1, 2, 3, 5. So, when condition number increases, the $j$ of TSD with the best performance increases.

We compute the average iteration numbers of Table \ref{result3} 
corresponding to different problem sets, and report the data in Table \ref{result3_3}. Furthermore, we plot a figure of the average iteration numbers and the corresponding parameter $j$, see Figure \ref{fig1}. From Figure \ref{fig1}, we can see that, the iteration numbers of problem set 1, 2, 3 have similar trend with the iteration number corresponding to all 5 problem sets, while the iteration numbers of problem set 4, 5 have a different trend. This reveals that the spectrum distribution of a strongly convex quadratic problem affects the performance of TSD. From the line corresponding to "all sets", we can see that, for the 5 sets, overall, TSD performs the best when j=50. Furthermore, its performance gets worse rapidly when $j$ is larger. Therefore, the results imply that smaller $j$ leads to better performance in this example. 

\begin{table}
\caption{Average number of iterations required by the TSD methods on the five problem sets in Subsection \ref{4_1}}
\label{result3}
\tiny
\begin{tabular}{|c|c|c|c|c|c|c|c|c|}
\hline  $\kappa$  & $\epsilon$ & TSD10$^1$ & TSD50& TSD100 & TSD200  &TSD300&TSD500&TSD1000 \\
\hline \multicolumn{9}{|c|}{problem set 1} \\
\hline \multirow{3}{*}{ $10^{4}$ } &  $10^{-6}$  &676.5&908.2&1986&1847.2&2708.7&3950.9&4398.6\\
 &  $10^{-9}$  &1291.7&1708.2&2582.9&3554.5&4437.4&6615.8&7928.4\\ 
 &  $10^{-12}$  &1553.9&2510.3&2735.7&4132.6&5123.9&8483.2&9973.9 \\
\hline \multirow{3}{*}{ $10^{5}$ } &  $10^{-6}$  &1044.9&1454.7&3034.3&3414.5&3141.5&6766.7&8701.3\\
  &  $10^{-9}$  &3487.8&2550.2&2937.9&6100.8& 8506.9&10044.8&18450.4 \\
  &  $10^{-12}$  &6611.5&2551.6&3771.9&6202.1 &8659.5&13760.3&17114.6\\
\hline \multirow{3}{*}{ $10^{6}$ } &  $10^{-6}$  &636.3&960.8&2059.6&2153.4&2508.3& 4423.3& 4439.5\\
  &  $10^{-9}$  &6081.6&2587.7& 3543.3&8661.8&17131.6&18719.2&25636.4 \\
  &  $10^{-12}$  &9482.1&3384.3&4403.7&8341.5&10360.3& 15468.6&25533.5\\
\hline \multirow{3}{*}{ $10^{10}$ } &  $10^{-6}$  &644.5&959.9&2040&2141.6&2566.6&4294.2&4437\\
  &  $10^{-9}$  &1729.2&1846.5&2512.2&4929.2&5683.4&9050.1&10863.2\\
  &  $10^{-12}$  &22905.9&4316.9&5336.3&7322.7&10515.6&19222.8&24150.1\\
\hline\multicolumn{9}{|c|}{problem set 2} \\
\hline \multirow{3}{*}{ $10^{4}$ } &  $10^{-6}$  &889.2&1543.9&1781.7&3587.1&4023.7&3956.6&4948.6\\
  &  $10^{-9}$  &1364.2&1698.2&2565&4341.4&4913.1&7979.6&10213.6\\ 
  &  $10^{-12}$  &2205.6&2252.3&2748.4&4911.5&7058.7&7849.4&16200 \\
\hline \multirow{3}{*}{ $10^{5}$ } &  $10^{-6}$  &1583.6&2006.8&2849.1&5553.7&5930.6&7817.6&10062.7\\
  &  $10^{-9}$  &2710.6&2141.7&3912.2&6917.7& 10770.6&15196.8&27449.4 \\
  &  $10^{-12}$  &7024.8&2938.4&4561.1&6611.4 &13932.5&13394.3&26879.2\\
\hline \multirow{3}{*}{ $10^{6}$ } &  $10^{-6}$  &1149.9&2070.2&2570&4331.5&4656.6&5586.6&8132.7\\
  &  $10^{-9}$  &6837.7&3661.4&5193.4&8296.1&17271.5&15597.4&35428.3\\
  &  $10^{-12}$  &8194.5&6999.8&4750.6&9841.6&15184.4&19786&34674.7\\
\hline \multirow{3}{*}{ $10^{10}$ } &  $10^{-6}$  &1063.8&2071.3&2585.6&4326.7&4667.7&5551.8&8177.6\\
  &  $10^{-9}$  &3153.7&2844.2&3458.4&8920.7&15350.9&13102.5&27804.3\\
  &  $10^{-12}$  &13165.4&7561.7&10123.4&15992.8&26555&34211.4&46405.6\\
\hline\multicolumn{9}{|c|}{problem set 3} \\
\hline \multirow{3}{*}{ $10^{4}$ } &  $10^{-6}$  &586&643.6&1935.5&2545.2&2911.9&4169&5201.7\\
  &  $10^{-9}$  &992&1083&2794.8&3829.9&5368.1&5748.7&8963.6\\ 
  &  $10^{-12}$  &1255.3&1389&4001.2&3486.2&5850.5&7557.5&12555.3\\
\hline \multirow{3}{*}{ $10^{5}$ } &  $10^{-6}$  &544.9&959.6&1220.7&2543.6&3857.8&6057.7&13858.9\\
  &  $10^{-9}$  &1249.3&2099.8&2003.7&5849.3&6216.8&13492.2&27481.4\\
  &  $10^{-12}$  &2162.4&1367.4&2358.2&5338.5&6644.7&10285.8&40796.4\\
\hline \multirow{3}{*}{ $10^{6}$ } &  $10^{-6}$  &664&899.3&1117.7&2152.5&2845.9&5213.3&10080.1\\
  &  $10^{-9}$  &3755.7&2255.1&2425.9&4766&6736.6&12406.6&23000.9\\
  &  $10^{-12}$  &5536.8&2281.4&3830.8&6714.8&10604.7&12994.5&20852.2\\
\hline \multirow{3}{*}{ $10^{10}$ } &  $10^{-6}$  &11.9&13&13&13&13&13&13\\
  &  $10^{-9}$  &27723.5&17579.9&9524.4&8043.3&11114.7&14753.9&14805.3\\
  &  $10^{-12}$  &68232.2&88600.7&79426.8&74621.2&93690.3&*$^2$&*\\
\hline\multicolumn{9}{|c|}{problem set 4} \\
\hline \multirow{3}{*}{ $10^{4}$ } &  $10^{-6}$  &1619.6&2191.5&1999.3&2871.6&3318&3962.7&3811.6\\
  &  $10^{-9}$  &1894.5&3396.1&1911.4&4403.5&4182.7&	4599.5&7024.9\\ 
  &  $10^{-12}$  &2761.7&3508.2&2716.2&6037.6&5097.6&5953.7&8256.7\\
\hline \multirow{3}{*}{ $10^{5}$ } &  $10^{-6}$  &2054.8&3378.2&4292.2&4621.1&3760.2&7583.5&6826.4\\
  &  $10^{-9}$  &4201.2&4301.9&3382.6&6104.2&6136&8199.1&10825.5\\
  &  $10^{-12}$  &5616.2&12963.1&7461.6&7596.8&9530.3&8079.2&15844.4\\
\hline \multirow{3}{*}{ $10^{6}$ } &  $10^{-6}$  &3753.9&4843&6018.6&5906.2&6322.4&6735.9&9119.3\\
  &  $10^{-9}$  &10507.3&17514&22655&41992.9&35490.5&24991.4&25421.2\\
  &  $10^{-12}$  &14916&44738.9&56434.5&74172&65703.9&58601.3&40133.8\\
\hline \multirow{3}{*}{ $10^{10}$ } &  $10^{-6}$  &9313.2&3770.5&7262&5024.9&5932.8&5343&18663.2\\
  &  $10^{-9}$  &59696.7&89058.7&88372.1&94300.2&*&*&96656.9\\
  &  $10^{-12}$  &79538.5&94545.2&*&*&*&*&*\\
\hline\multicolumn{9}{|c|}{problem set 5} \\
\hline \multirow{3}{*}{ $10^{4}$ } &  $10^{-6}$  &582&1038.8&2420.6&1739.3&2410.2&2914.9&4038.9\\
  &  $10^{-9}$  &1084.5&1415.5&3454.5&2564&3341.6&3701.1&5520.8\\ 
  &  $10^{-12}$  &1336.2&1905.4&4419.2&3217.4&4540.6&5184.1&6726.3\\
\hline \multirow{3}{*}{ $10^{5}$ } &  $10^{-6}$  &528&895.6&1342.3&2247.6&2471.1&4216.9&6021.5\\
  &  $10^{-9}$  &1323.6&1981.4&3566.5&4641.7&4724.6&6073.4&10917.4\\
  &  $10^{-12}$  &1994.9&3155.5&4877.7&8699.6&6875.3&9182&14118.3\\
\hline \multirow{3}{*}{ $10^{6}$ } &  $10^{-6}$  &913&905.1&1146.4&1561.5&2766.7&4163.6&6866\\
  &  $10^{-9}$  &4063&3558.1&5852.1&9564.8&12577&17867.8&17106.2\\
  &  $10^{-12}$  &5484.9&3961.4&5578.2&10421.6&11581.4&18540.7&22034\\
\hline \multirow{3}{*}{ $10^{10}$ } &  $10^{-6}$  &11.4&12.3&12.3&12.3&12.3&12.3&12.3\\
  &  $10^{-9}$  &42714.9&25543.2&17093.4&6742.4&10232.4&7302.2&9167.5\\
  &  $10^{-12}$  &*&47466.5&94131.9&43662.2&70681.9&70590.7&*\\
\hline
\end{tabular}
\begin{tablenotes}%
\item[$^{1}$] TSD10 means the TSD method with $j=10$, the same as TSD50 and so on.
\item[$^{2}$] * means the max iteration number is reached.
\end{tablenotes}
\end{table}

\begin{table}
\caption{Average number of iterations in Table \ref{result3}}
\label{result3_3}
\begin{tabular}{|c|c|c|c|c|c|c|c|}
\hline \diagbox{set}{method}& TSD10 & TSD50& TSD100 & TSD200  &TSD300&TSD500&TSD1000\\
\hline set1 &4678.83&2144.94&3078.65&4900.16&6778.64&10066.66&13468.91\\
\hline set2 &4111.92&3149.16&3924.91&6969.35&10859.61&12502.50&21364.73\\
\hline set3 &9392.83&9930.98&9221.06&9991.96&12987.92&16057.68&23134.07\\
\hline set4 &16322.80&23684.11&25208.79&29419.25&28789.53&27837.44&28548.66\\
\hline set5 &13336.37&7653.23&11991.26&7922.87&11017.93&12479.14&16877.43\\
\hline all 5 sets &9568.55&9312.49&10684.93&11840.72&14086.73&15788.69&20678.76\\
\hline
\end{tabular}
\end{table}

\begin{figure*}[!ht]%
\centering
\includegraphics[width=\textwidth]{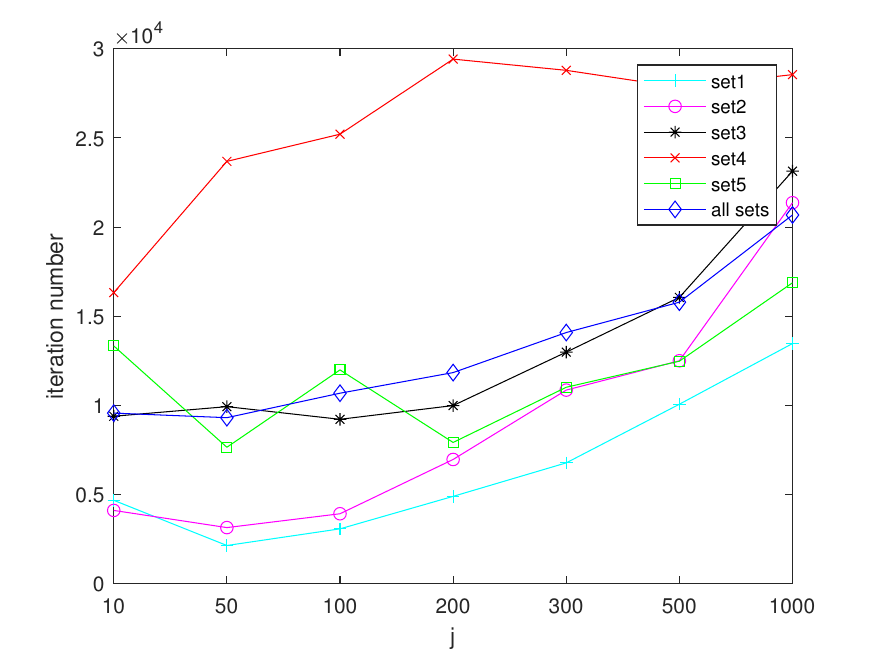}
\caption{Average iteration number vs. parameter $j$ }\label{fig1}
\end{figure*}

We also make Figure \ref{fig3}, "$\log_{10}\|x-x^*\|_2$ vs. iteration number", with using an example form Problem set 1 with $\kappa=1e4$, $\epsilon=1e-12$ as a numerical verification of Theorem \ref{SD2_con}. After the iteration number 500 or so, SD turns into a linear convergence rate, while TSD methods show superlinear convergence rate. This result is consistent with the conclusion of Theorem \ref{SD2_con}, it also reveals that TSD's actual converge property is much better and can be further explored.
\begin{figure*}[!ht]%
\centering
\includegraphics[width=\textwidth]{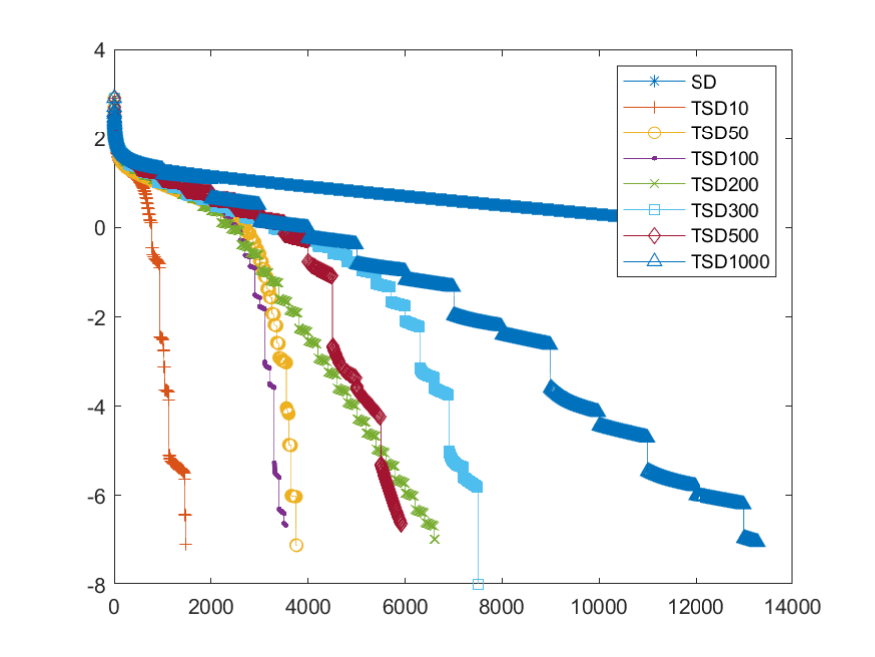}
\caption{$\log_{10}\|x-x^*\|_2$ vs. iteration number }\label{fig3}
\end{figure*}

\subsection{Comparison with other competitive gradient methods}
\label{comparison_result}
We compared the method (3.2) with other recent successful gradient methods including  \cite{Huang2022a}:

(i) BB \cite{Barzilai1988}: the original BB method using  $\alpha_{k}^{B B 1}$;

(ii) DY \cite{Dai2005}: the monotone Dai-Yuan gradient method (formula (5.3) in \cite{Dai2005});

(iii) ABBmin 2 \cite{Frassoldati2008}: a gradient method adaptively using  $\alpha_{k}^{B B 1}$  and a short stepsize in \cite{Frassoldati2008};

(iv) BBQ \cite{Huang2021}: a gradient method adaptively using  $\alpha_{k}^{B B 1}$  and the short stepsize  $\min \left\{\alpha_{k-1}^{B B 2}, \alpha_{k}^{B B 2}, \alpha_{k}^{B B Q}\right\}$.

We test the same problem as in last subsection \ref{4_1}. 
The parameters of each method are chosen to achieve their best performance in our test. In particular, for ABBmin2 \cite{Frassoldati2008}, we use $\tau\in\{0.1,0.2, \ldots, 0.9\}$. For BBQ \cite{Huang2021}, we chose  $\tau$  from  $\{0.1,0.2, \ldots, 0.9\}$ for each given  $\gamma \in\{1,1.02,$ $1.05,1.1,1.2,1.3\}$. 

 We set  $\kappa\in\{10^{4}, 10^{5}, 10^{6}\}$, max iteration number is set to $5e4$, and other settings are not changed. The results are presented in Table \ref{result1}.

We can see that our method TSD with $j\in\{10, 50, 100\}$ performs better than BB and DY in problem set 1, 2, 3, 5, comparable to them in problem set 4 in the sense of number of iterations, and do not better than the ABBmin2 and BBQ methods in all problem sets in the sense of number of iterations. It shows that the TSD method is an effective gradient method, its performance has something to do with the spectral distribution of a quadratic problem.

\begin{table}
\caption{Average number of iterations required by the compared methods on the five problem sets in Subsection \ref{4_1}}
\label{result1}
\small
\begin{tabular}{|c|c|c|c|c|c|c|c|c|}
\hline  $\kappa$  & $\epsilon$ & BB1 & DY & ABBmin2 & BBQ  &TSD10&TSD50&TSD100 \\
\hline \multicolumn{9}{|c|}{problem set 1} \\
\hline \multirow{3}{*}{ $10^{4}$ } &  $10^{-6}$  &1017	&497.5&403.2&381.2&676.5&908.2&1986\\
  &  $10^{-9}$  &1625.4&950.2&744.7&747.4&1291.7&1708.2&2582.9\\ 
  &  $10^{-12}$  &2136.5&1419.9&1007.7&1048.3&1553.9&2510.3&2735.7\\
\hline \multirow{3}{*}{ $10^{5}$ } &  $10^{-6}$  &2405.9&1004.9&657.5&634.7&1044.9&1454.7&3034.3\\
  &  $10^{-9}$  &4162&2680.2&1205.4&1377.7& 3487.8&2550.2&2937.9 \\
  &  $10^{-12}$  &5078.6&4341.4&1437.9&1776.8 &6611.5&2551.6&3771.9\\
\hline \multirow{3}{*}{ $10^{6}$ } &  $10^{-6}$  &5759.1&550.9& 457.6 &428.4&636.3&960.8&2059.6\\
  &  $10^{-9}$  & 10067.4& 8545& 1231.9& 1630.7&6081.6&2587.7& 3543.3 \\
  &  $10^{-12}$  &10915& 13538.4&1434.6&2059.4&9482.1&3384.3&4403.7\\
\hline \multicolumn{9}{|c|}{problem set 2} \\
\hline \multirow{3}{*}{ $10^{4}$ } &  $10^{-6}$  & 1051.3& 475.9&397.1&386.9&889.2&1543.9&1781.7\\
  &  $10^{-9}$  &1583.1&909.9&734.4&712.3 & 1364.2&1698.2&2565\\
  &  $10^{-12}$  & 2237.8& 1365.5&984.6&889.6&2205.6&2252.3&2748.4\\
\hline \multirow{3}{*}{ $10^{5}$ } &  $10^{-6}$  & 3702.2& 1018.4& 761.8&662.1&1583.6&2006.8&2849.1\\
  &  $10^{-9}$  &6513.7&2970.1& 1550.8&1424.4&2710.6&2141.7&3912.2\\
  &  $10^{-12}$  &7413.3&4401.6&1818.5&1756.3&7024.8&2938.4&4561.1\\
\hline \multirow{3}{*}{ $10^{6}$ } &  $10^{-6}$  &14220.9&661.4&579.5&533.8 &1149.9&2070.2&2570\\
  &  $10^{-9}$  &23090.1&9208.9&1841.8&1882.1&6837.7&3661.4&5193.4\\
  &  $10^{-12}$  &23531.3&15261.7&2201.3&2272.2&8194.5&6999.8&4750.6\\
\hline \multicolumn{9}{|c|}{problem set 3} \\
\hline \multirow{3}{*}{ $10^{4}$ } &  $10^{-6}$  &662.8& 332.6&243.6&231.4&586&643.6&1935.5\\
  &  $10^{-9}$  & 1028.3&708.7&436.3&442.6&992&1083&2794.8\\
  &  $10^{-12}$  &1386.2& 978.8&611.7&566.4&1255.3&1389&4001.2\\
\hline \multirow{3}{*}{ $10^{5}$ } &  $10^{-6}$  & 1365.2&696.8&249.4&244.2&544.9&959.6&1220.7\\
  &  $10^{-9}$  & 2489.1& 1823.9&521.9&550.4&1249.3&2099.8&2003.7\\
  &  $10^{-12}$  &3126.5&2570.3&708.7&639.6&2162.4&1367.4&2358.2\\
\hline \multirow{3}{*}{ $10^{6}$ } &  $10^{-6}$  & 3047& 517.3& 79.6&138.8&664&899.3&1117.7\\
  &  $10^{-9}$  & 5005.6&3369.6&555.7&528.6&3755.7&2255.1&2425.9 \\
  &  $10^{-12}$  &6527.7&4722.2&678.1&562.7&5536.8&2281.4&3830.8\\
\hline \multicolumn{9}{|c|}{problem set 4} \\
\hline \multirow{3}{*}{ $10^{4}$ } &  $10^{-6} $ & 982.6&553.6&494.7&485.2&1619.6&2191.5&1999.3\\
  &  $10^{-9}$  & 1658.2& 1047.4& 879.3 &920.6&1894.5&3396.1&1911.4\\
  &  $10^{-12}$  &2123.1&1471.4&1186.4&1182.8&2761.7&3508.2&2716.2\\
\hline \multirow{3}{*}{ $10^{5}$ } &  $10^{-6}$  & 2578.2& 1288.6&1114.8&1050.2&2054.8&3378.2&4292.2\\
  &  $10^{-}$  & 4537.5&3262.7&2114.6&2241.2&4201.2&4301.9&3382.6\\
  &  $10^{-12}$  &5650.2&5033& 2334& 2610.3& 5616.2&12963.1&7461.6\\
\hline \multirow{3}{*}{ $10^{6}$ } &  $10^{-6}$  &6802.3&1849.5&1565&1450.9&3753.9&4843&6018.6\\
  &  $10^{-9}$  & 12200.5&11987.5&3126.7&4078.2&10507.3&17514&22655\\
  &  $10^{-12}$  &12585.7& 16035.8&2993.8&4616.9&14916&44738.9&56434.5\\
\hline \multicolumn{9}{|c|}{problem set 5} \\
\hline \multirow{3}*{ $10^{4}$ } &  $10^{-6}$  & 581.9&348.8&248.9&237.8&582&1038.8&2420.6\\
  &  $10^{-9}$  &955.7& 702.2&469.2&442.5&1084.5&1415.5&3454.5\\
  &  $10^{-12}$  &1275.4&1030.3& 612.6&565.2&1336.2&1905.4&4419.2\\
\hline \multirow{3}*{ $10^{5}$ } &  $10^{-6}$  & 1218.5&762.5&258.1&249.6&528&895.6&1342.3\\
  &  $10^{-9}$  &2243.1&1745.1&549&557.1&1323.6&1981.4&3566.5\\
  &  $10^{-12}$  &2930.5&2573.9&706.6&596&1994.9&3155.5&4877.7\\
\hline \multirow{3}*{ $10^{6}$ } &  $10^{-6}$  &2460.5&612.9&87.7&149.8&913&905.1&1146.4\\
  &  $10^{-9}$  &4953.6&3483.5&519.4&493.9&4063&3558.1&5852.1\\
  &  $10^{-12}$  &5776.5&4634.6&629.1&517&5484.9&3961.4&5578.2\\
\hline
\end{tabular}
\end{table}

\subsection{Results on general unconstrained problems from the CUTEst collection}\label{CUTEST}
In this subsection, We report numerical results of Algorithm \ref{TSDaluo} (the TSD algorithm) on general unconstrained problems from the CUTEst collection \cite{Gould2014} with dimension less than or equal to 10000, whose total number is 280. As shown in the before two subsections, only BBQ and ABBmin \cite{Frassoldati2008,diSerafino2018} perform better than TSD in our example, generally parameter 10, 50 work better for TSD. As shown in \cite{Huang2021}, the BBQ method is much faster than the ABBmin method in these unconstrained problems from the CUTEst collection.     
 Hence, we only compared the TSD algorithm with the BBQ algorithm for unconstrained optimization, and use $j\in\{10, 50\}$ for the TSD algorithm only. To this end, we use the Julia interface of CUTEst\footnote{Refer to: https://github.com/JuliaSmoothOptimizers/CUTEst.jl}, rewrite the BBQ algorithm \footnote{Downloaded form: https://lsec.cc.ac.cn/~dyh/software.html} code in the Julia language, and implemented the methods by Julia (version 1.11.5). 

The termination condition used for solving a problem here is $\|g\|_{\infty}<10^{-6}$. There are 3 problems in the problem set in this subsection satisfies the termination condition from the beginning, corresponding to problem names "FLETCBV2", "S308NE", "MOREBV", so we do not report the results of them. We use the default setting of the BBQ algorithm, it also terminates if its iterations over 200,000, or its function evaluation exceeds one million, or its Grippo-Lampariello-Lucidi (GLL) nonmonotone line search fails. Similarly, the TSD algorithm also terminates, for both $j=10$ and $j=50$, if its iterations over 10,000. In the situation that a method terminates without satisfying the termination condition, we treated this method as failed. We deleted the problems if it can not be solved by one of the three methods, and 182 problems are left. 

\begin{figure}
\caption{Performance profile}
\label{fig2}
\includegraphics[width=\textwidth]{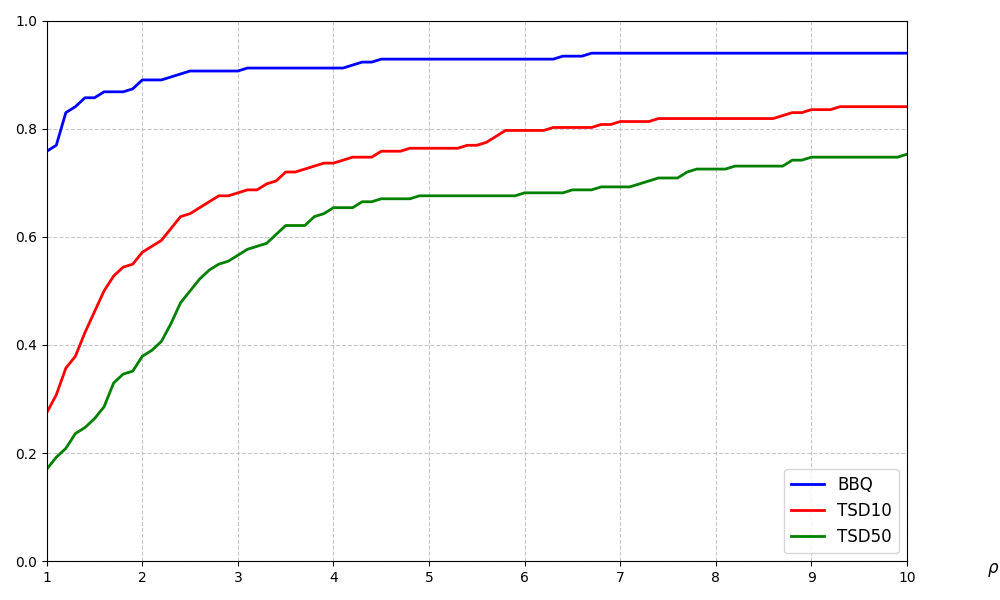}
\end{figure}

The strong Wolfe condition line search used in the TSD algorithm are implemented by the classical way as Algorithm 3.5 and Algorithm 3.6 in \cite{JorgeNocedal2006}. Therefore, every line search used by the TSD algorithm will use both function evaluations and gradient evaluations, while every line search used by BBQ, which is Grippo-Lampariello-Lucidi (GLL) nonmonotone line search, use only function evaluations. So, every iteration of the TSD algorithm will naturally takes longer time than every iteration of the BBQ algorithm. Hence, we only make the performance profile \cite{Dolan2001} \footnote{Performance profiles is the distribution functions for a certain performance metric which is a great tool for comparing optimization methods or software. \cite{Dolan2001}.}of the TSD algorithm and the BBQ algorithm using iteration metric as plotted in Figure \ref{fig2}. For each method, the vertical axis of the figure shows the percentage of problems the method solves within the factor $\rho$ of the minimum value of the iteration metric. From Figure \ref{fig2}, it can be seen that the TSD algorithm performs not better than the BBQ algorithm. 

In Tables \ref{A_table_start}-\ref{A_table_last} in Appendix \ref{A_tables}, we present the number of iterations “iter”, and the CPU time in seconds “time” costed by the TSD algorithm and the BBQ algorithm, which show that the TSD algorithm is not faster than the BBQ algorithm on most of the problems.
%

Numerical experiments demonstrate that TSD, as a first-step formalization of the direction-aggregation idea, already shows competitive and sometimes superior performance to well-established methods like BB and DY in ill-conditioned quadratic problems. While its current performance on general unconstrained problems is not yet superior to state-of-the-art stepsize-based methods like BBQ, it opens promising directions for future research. A promising avenue is developing adaptive strategies for the parameter $j$.

\section{Conclusions}
\label{conclusion}
In this paper, we propose the TSD method, leveraging SD's geometric property which can reach the minimum in 3 iterations in quadratic problems. To the best of our knowledge, this is the first formalization of this idea since Akaike (1959)  \cite{Akaike1959}. 
The gradient sequence generated by the  TSD method is proved to be R-linearly convergent for strongly convex quadratic problems. 
Numerical experiments show TSD has superlinear convergence behavior for strongly convex quadratic problems, TSD outperforms the BB and DY method in quadratic cases and TSD has potential for its success in general unconstrained optimization.
Future work will focus on extending the convergence theory of TSD to general non-quadratic objectives, exploring adaptive strategies for the cycle parameter $j$.

\section*{Declarations}

\begin{itemize}
\item Funding: This work was supported by National Natural Science Foundation of China grants (No. 11271032 and No. 12271526).
\item Competing interests: The authors have no relevant financial or non-financial interests to disclose.
\item Code availability: Most of the codes in this paper are available at https://github.com/ya-shen200012/TSD/releases/tag/v1.1.
\end{itemize}

\begin{appendices}

\section{Tables}\label{A_tables}
Tables for subsection \ref{CUTEST}.
\begin{table}
\caption{Results of the TSD algorithm and the BBQ algorithm on 182 unconstrained problems from the CUTEst collection}
\label{A_table_start}
\begin{tabular}{|l|l|ll|ll|ll|}
\hline \multirow[c]{2}{*}{problem} & \multirow[c]{2}{*}{ n } & \multicolumn{2}{|c|}{BBQ} & \multicolumn{2}{|c|}{TSD10$^{1}$}  & \multicolumn{2}{|c|}{TSD50}\\
 
& & iter & time & iter & time & iter & time 
\\\hline ALLINITU&4&15&0.0&15&0.0&18&0.0
\\\hline ARGLINA&200&2&0.002&1&0.001&1&0.0
\\\hline ARGTRIGLS&200&1234&0.469&2363&6.082&10000&25.222
\\\hline ARWHEAD&5000&3&0.002&704&1.419&10000&19.358
\\\hline BA-L1LS&57&31&0.0&61&0.01&67&0.009
\\\hline BARD&3&72&0.001&399&0.003&1703&0.013
\\\hline BDQRTIC&5000&74&0.04&3984&11.624&10000&24.162
\\\hline BEALE&2&39&0.0&53&0.0&252&0.002
\\\hline BIGGS6&6&294&0.001&1832&0.042&4653&0.103
\\\hline BOX&10000&25&0.04&9671&83.821&10000&81.392
\\\hline BOX3&3&30&0.0&174&0.002&366&0.004
\\\hline BOXBODLS&2&7&0.0&10000&0.2&10000&0.141
\\\hline BRKMCC&2&8&0.0&24&0.0&89&0.001
\\\hline BROWNAL&200&49069&6.541&233&0.244&869&0.918
\\\hline BROWNDEN&4&62&0.0&541&0.012&7649&0.139
\\\hline BROYDN3DLS&5000&80&0.025&36&0.055&58&0.094
\\\hline BROYDN7D&5000&2489&4.853&2693&34.451&4154&49.762
\\\hline BROYDNBDLS&5000&50&0.031&42&0.115&52&0.147
\\\hline BRYBND&5000&50&0.027&42&0.12&52&0.146
\\\hline CHAINWOO&4000&574&0.211&796&1.913&1401&3.163
\\\hline CHNROSNB&50&887&0.005&1428&0.038&2060&0.055
\\\hline CHNRSNBM&50&1033&0.007&1807&0.049&2352&0.064
\\\hline CHWIRUT1LS&3&168&0.004&2&0.0&2&0.0
\\\hline CHWIRUT2LS&3&66&0.001&2&0.0&2&0.0
\\\hline CLIFF&2&36&0.001&172&0.001&660&0.003
\\\hline CLUSTERLS&2&18&0.0&16&0.0&26&0.001
\\\hline COSINE&10000&21&0.035&10000&161.47&10000&158.436
\\\hline CRAGGLVY&5000&170&0.127&284&1.535&376&1.792
\\\hline CUBE&2&139&0.0&363&0.003&1063&0.006
\\\hline DANIWOODLS&2&23&0.0&1&0.0&1&0.0
\\\hline DENSCHNA&2&14&0.0&21&0.0&43&0.0
\\\hline DENSCHNB&2&8&0.0&13&0.0&13&0.0
\\\hline DENSCHNC&2&6&0.0&32&0.0&76&0.0
\\\hline DENSCHNE&3&47&0.0&12&0.0&11&0.0
\\\hline DENSCHNF&2&17&0.0&21&0.0&24&0.0
\\\hline DIXMAANA1&3000&7&0.002&7&0.003&7&0.009
\\\hline DIXMAANB&3000&7&0.002&6&0.003&6&0.003
\\\hline DIXMAANC&3000&8&0.002&7&0.007&7&0.005
\\\hline DIXMAAND&3000&9&0.002&8&0.109&8&0.005
\\\hline DIXMAANE1&3000&357&0.048&411&0.242&752&0.473
\\\hline DIXMAANF&3000&293&0.057&453&0.379&702&0.625
\\\hline DIXMAANG&3000&346&0.064&682&0.608&820&0.771
\\\hline DIXMAANH&3000&242&0.041&624&0.522&650&0.594
\\\hline DIXMAANI1&3000&2307&0.304&8678&5.702&6852&4.273
\\\hline DIXMAANJ&3000&303&0.056&495&0.441&675&0.645
\\\hline
\end{tabular}
\begin{tablenotes}%
\item[$^{1}$] Here, TSD10 means the TSD algorithm with $j=10$, the same as TSD50 and following tables. 
\end{tablenotes}
\end{table}
\begin{table}
\caption{Results of the TSD algorithm and the BBQ algorithm on 182 unconstrained problems from the CUTEst collection}
\begin{tabular}{|l|l|ll|ll|ll|}
\hline \multirow{2}{*}{problem} & \multirow{2}{*}{ n } & \multicolumn{2}{|c|}{BBQ} & \multicolumn{2}{|c|}{TSD10}  & \multicolumn{2}{|c|}{TSD50}\\
 & & iter & time & iter & time & iter & time 
\\\hline DIXMAANK&3000&312&0.065&367&0.318&652&0.57
\\\hline DIXMAANL&3000&220&0.042&327&0.281&517&0.445
\\\hline DIXMAANM1&3000&2904&0.395&10000&6.812&10000&6.303
\\\hline DIXMAANN&3000&702&0.141&1251&1.117&1814&1.648
\\\hline DIXMAANO&3000&747&0.152&1651&1.637&1802&1.803
\\\hline DIXMAANP&3000&605&0.104&1352&1.162&1456&1.411
\\\hline DIXON3DQ&10000&4854&1.54&10000&15.48&10000&13.845
\\\hline DQDRTIC&5000&14&0.006&81&0.186&102&0.219
\\\hline ECKERLE4LS&3&4&0.0&5&0.0&5&0.0
\\\hline EDENSCH&2000&39&0.009&46&0.05&54&0.058
\\\hline EG2&1000&5&0.001&13&0.02&95&0.103
\\\hline EGGCRATE&2&8&0.0&6&0.001&6&0.0
\\\hline EIGENBLS&2550&41398&176.381&10000&221.064&10000&213.936
\\\hline EIGENCLS&2652&63696&286.209&10000&258.136&10000&247.75
\\\hline ELATVIDU&2&18&0.0&36&0.0&46&0.0
\\\hline ENGVAL1&5000&24&0.012&38&0.088&48&0.105
\\\hline ENGVAL2&3&98&0.0&843&0.007&1852&0.015
\\\hline ENSOLS&9&92&0.015&143&0.192&224&0.304
\\\hline ERRINROS&50&2993&0.018&10000&0.282&10000&0.268
\\\hline EXP2&2&22&0.0&31&0.001&51&0.0
\\\hline EXPFIT&2&33&0.001&42&0.0&132&0.003
\\\hline EXTROSNB&1000&24814&1.317&10000&3.473&10000&3.279
\\\hline FLETBV3M&5000&48&0.072&10000&107.267&10000&99.211
\\\hline FLETCHCR&1000&341&0.03&903&0.493&2502&1.279
\\\hline FMINSRF2&5625&697&0.353&977&2.321&1156&2.577
\\\hline FMINSURF&5625&931&0.762&2189&5.264&1804&4.142
\\\hline FREUROTH&5000&59&0.031&10000&28.632&10000&29.527
\\\hline GAUSS1LS&8&3078&0.273&10000&4.065&10000&4.052
\\\hline GAUSS2LS&8&5349&0.427&203&0.114&252&0.143
\\\hline GAUSSIAN&3&5&0.0&12&0.0&52&0.002
\\\hline GBRAINLS&2&16&0.022&1115&12.609&1725&19.158
\\\hline GENHUMPS&5000&9313&17.843&10000&141.232&10000&140.006
\\\hline GENROSE&500&3152&0.123&4312&1.052&8599&2.003
\\\hline GROWTHLS&3&1&0.0&1&0.0&1&0.0
\\\hline GULF&3&351&0.03&10000&5.093&10000&4.78
\\\hline HAIRY&2&43&0.0&1635&0.023&1055&0.012
\\\hline HATFLDD&3&36&0.0&2422&0.036&3053&0.043
\\\hline HATFLDE&3&113&0.0&1048&0.026&1404&0.033
\\\hline HATFLDFL&3&167&0.0&10000&0.056&10000&0.069
\\\hline HATFLDFLS&3&455&0.001&6612&0.045&10000&0.067
\\\hline HATFLDGLS&25&177&0.0&715&0.01&752&0.009
\\\hline HEART8LS&8&750&0.003&1749&0.018&2942&0.039
\\\hline HELIX&3&63&0.0&134&0.001&552&0.005
\\\hline HILBERTA&2&7&0.0&31&0.0&54&0.001
\\\hline HILBERTB&10&7&0.0&8&0.0&8&0.0
\\\hline
\end{tabular}
\end{table}
\begin{table}
\caption{Results of the TSD algorithm and the BBQ algorithm on 182 unconstrained problems from the CUTEst collection}
\begin{tabular}{|l|l|ll|ll|ll|}
\hline \multirow{2}{*}{problem} & \multirow{2}{*}{ n } & \multicolumn{2}{|c|}{BBQ} & \multicolumn{2}{|c|}{TSD10}  & \multicolumn{2}{|c|}{TSD50}\\
& & iter & time & iter & time & iter & time 
\\\hline HIMMELBB&2&1&0.0&9&0.0&9&0.0
\\\hline HIMMELBCLS&2&15&0.0&12&0.001&21&0.0
\\\hline HIMMELBF&4&5402&0.016&10000&0.098&10000&0.103
\\\hline HIMMELBG&2&10&0.0&9&0.0&9&0.0
\\\hline HIMMELBH&2&12&0.0&16&0.0&17&0.0
\\\hline HUMPS&2&168&0.001&106&0.001&307&0.003
\\\hline INTEQNELS&502&7&0.017&6&0.066&6&0.074
\\\hline JENSMP&2&27&0.0&1&0.0&1&0.001
\\\hline JIMACK&3549&6346&191.317&10000&2204.867&10000&2204.652
\\\hline JUDGE&2&30&0.0&41&0.001&102&0.001
\\\hline JUDGEB&2&30&0.0&41&0.0&102&0.0
\\\hline KOWOSB&4&112&0.001&462&0.003&1115&0.009
\\\hline LANCZOS1LS&6&2254&0.019&10000&0.372&10000&0.367
\\\hline LANCZOS2LS&6&2928&0.023&10000&0.383&10000&0.354
\\\hline LANCZOS3LS&6&2590&0.02&10000&0.374&10000&0.345
\\\hline LIARWHD&5000&47&0.034&2022&6.483&9153&24.407
\\\hline LOGHAIRY&2&267&0.0&974&0.01&4276&0.045
\\\hline LRA9A&123&127&0.976&144&6.264&302&12.033
\\\hline LRIJCNN1&22&54&0.611&62&6.792&106&11.299
\\\hline LRW1A&300&5123&3.025&34&0.172&52&0.282
\\\hline LRW8A&300&4434&47.292&10000&568.87&10000&522.239
\\\hline LSC1LS&3&545&0.002&281&0.002&773&0.005
\\\hline LUKSAN11LS&100&4994&0.052&3698&0.216&9642&0.547
\\\hline LUKSAN12LS&98&348&0.005&773&0.082&2656&0.29
\\\hline LUKSAN13LS&98&222&0.003&474&0.053&824&0.063
\\\hline LUKSAN14LS&98&194&0.005&300&0.014&577&0.028
\\\hline LUKSAN15LS&100&31&0.004&34&0.039&53&0.06
\\\hline LUKSAN16LS&100&37&0.002&42&0.02&55&0.025
\\\hline LUKSAN17LS&100&539&0.136&832&1.661&1359&2.736
\\\hline LUKSAN21LS&100&1726&0.016&2482&0.087&3402&0.113
\\\hline LUKSAN22LS&100&9843&0.252&10000&1.69&10000&1.708
\\\hline MANCINO&100&14&0.059&10000&349.59&10000&336.397
\\\hline MEYER3&3&128783&1.59&7792&0.209&5820&0.162
\\\hline MGH09LS&4&91&0.0&223&0.002&10000&0.067
\\\hline MGH17LS&5&11643&0.101&10000&0.41&10000&0.349
\\\hline MISRA1ALS&2&152698&1.579&12&0.0&19&0.001
\\\hline MNISTS0LS&494&1&0.061&1&0.181&1&0.183
\\\hline MNISTS5LS&494&1&0.06&1&0.18&1&0.179
\\\hline MSQRTALS&1024&4324&2.614&10000&32.376&10000&30.785
\\\hline MSQRTBLS&1024&3300&2.102&10000&32.671&10000&31.036
\\\hline MUONSINELS&1&35&0.002&10000&7.776&3973&3.088
\\\hline NCB20&5010&417&0.67&10000&94.405&10000&90.689
\\\hline NCB20B&5000&2933&4.361&10000&97.128&10000&95.492
\\\hline NELSONLS&3&200000&4.365&3&0.0&3&0.0
\\\hline NONCVXU2&5000&18483&20.813&10000&68.398&10000&65.306
\\\hline
\end{tabular}
\end{table}
\begin{table}
\caption{Results of the TSD algorithm and the BBQ algorithm on 182 unconstrained problems from the CUTEst collection}
\begin{tabular}{|l|l|ll|ll|ll|}
\hline \multirow{2}{*}{problem} & \multirow{2}{*}{ n } & \multicolumn{2}{|c|}{BBQ} & \multicolumn{2}{|c|}{TSD10}  & \multicolumn{2}{|c|}{TSD50}\\
 & & iter & time & iter & time & iter & time 
\\\hline NONDIA&5000&24&0.014&10000&22.655&10000&22.488
\\\hline NONDQUAR&5000&2675&0.63&5504&5.676&10000&9.335
\\\hline OSBORNEB&11&701&0.018&1934&0.289&3014&0.445
\\\hline OSCIPATH&500&12&0.001&14&0.003&15&0.003
\\\hline PALMER5C&6&27&0.0&31&0.001&60&0.0
\\\hline PALMER5D&4&78&0.001&1840&0.028&2426&0.031
\\\hline POWELLSG&5000&261&0.067&3516&3.618&6886&6.893
\\\hline POWELLSQLS&2&18766&0.06&8&0.0&8&0.0
\\\hline POWERSUM&4&8&0.0&7&0.0&7&0.0
\\\hline PRICE3&2&41&0.0&32&0.0&102&0.0
\\\hline PRICE4&2&14&0.0&103&0.001&603&0.003
\\\hline QING&100&100&0.001&92&0.002&152&0.004
\\\hline RAT43LS&4&3&0.0&285&0.005&51&0.001
\\\hline RECIPELS&3&37&0.0&32&0.0&154&0.001
\\\hline ROSENBR&2&92&0.001&132&0.0&552&0.002
\\\hline ROSENBRTU&2&124&0.0&1660&0.011&5252&0.036
\\\hline S308&2&13&0.0&13&0.0&17&0.0
\\\hline SCHMVETT&5000&62&0.086&10000&116.314&10000&109.082
\\\hline SENSORS&100&25&0.106&590&22.667&66&2.197
\\\hline SINEVAL&2&167&0.001&542&0.006&2303&0.023
\\\hline SINQUAD&5000&74&0.072&10000&71.218&10000&69.761
\\\hline SINQUAD2&5000&256&0.281&10000&80.639&10000&77.812
\\\hline SISSER&2&12&0.0&5&0.0&5&0.0
\\\hline SISSER2&2&14&0.0&7&0.0&7&0.0
\\\hline SNAIL&2&8&0.0&5&0.0&5&0.001
\\\hline SPARSINE&5000&9895&9.734&10000&99.151&10000&103.729
\\\hline SPARSQUR&10000&31&0.036&23&0.111&45&0.284
\\\hline SPIN2LS&102&120&0.02&18&0.014&18&0.016
\\\hline SPMSRTLS&4999&584&0.269&732&1.822&656&1.524
\\\hline SROSENBR&5000&16&0.004&112&0.152&410&0.531
\\\hline STRTCHDV&10&68&0.001&22&0.002&48&0.003
\\\hline TESTQUAD&10&189&0.0&10000&0.105&10000&0.1
\\\hline TOINTGOR&50&198&0.004&348&0.022&352&0.021
\\\hline TOINTGSS&5000&1&0.0&1&0.001&1&0.002
\\\hline TOINTPSP&50&203&0.001&321&0.007&313&0.009
\\\hline TOINTQOR&50&57&0.001&80&0.002&128&0.002
\\\hline TQUARTIC&5000&26&0.015&10000&16.827&10000&15.443
\\\hline TRIDIA&5000&1480&0.246&10000&12.932&10000&13.379
\\\hline TRIGON1&10&80&0.001&132&0.004&202&0.005
\\\hline TRIGON2&10&29&0.0&163&0.005&206&0.007
\\\hline VAREIGVL&5000&270&0.198&262&1.002&442&1.585
\\\hline WATSON&12&789&0.007&10000&0.297&10000&0.285
\\\hline WAYSEA1&2&36&0.0&203&0.001&803&0.006
\\\hline WAYSEA2&2&23&0.0&63&0.0&202&0.002
\\\hline WOODS&4000&219&0.052&297&0.381&352&0.697
\\\hline
\end{tabular}
\end{table}
\begin{table}
\caption{Results of the TSD algorithm and the BBQ algorithm on 182 unconstrained problems from the CUTEst collection}
\label{A_table_last}
\begin{tabular}{|l|l|ll|ll|ll|}
\hline \multirow{2}{*}{problem} & \multirow{2}{*}{ n } & \multicolumn{2}{|c|}{BBQ} & \multicolumn{2}{|c|}{TSD10}  & \multicolumn{2}{|c|}{TSD50}\\
 & & iter & time & iter & time & iter & time 
\\\hline YATP2LS&8&49&0.001&11&0.0&12&0.0
\\\hline ZANGWIL2&2&1&0.0&1&0.0&1&0.0

\\\hline
\end{tabular}
\end{table}

\end{appendices}

\bibliography{jsc_ref.bib}


\begin{thebibliography}{29}
\ifx \bisbn   \undefined \def \bisbn  #1{ISBN #1}\fi
\ifx \binits  \undefined \def \binits#1{#1}\fi
\ifx \bauthor  \undefined \def \bauthor#1{#1}\fi
\ifx \batitle  \undefined \def \batitle#1{#1}\fi
\ifx \bjtitle  \undefined \def \bjtitle#1{#1}\fi
\ifx \bvolume  \undefined \def \bvolume#1{\textbf{#1}}\fi
\ifx \byear  \undefined \def \byear#1{#1}\fi
\ifx \bissue  \undefined \def \bissue#1{#1}\fi
\ifx \bfpage  \undefined \def \bfpage#1{#1}\fi
\ifx \blpage  \undefined \def \blpage #1{#1}\fi
\ifx \burl  \undefined \def \burl#1{\textsf{#1}}\fi
\ifx \doiurl  \undefined \def \doiurl#1{\url{https://doi.org/#1}}\fi
\ifx \betal  \undefined \def \betal{\textit{et al.}}\fi
\ifx \binstitute  \undefined \def \binstitute#1{#1}\fi
\ifx \binstitutionaled  \undefined \def \binstitutionaled#1{#1}\fi
\ifx \bctitle  \undefined \def \bctitle#1{#1}\fi
\ifx \beditor  \undefined \def \beditor#1{#1}\fi
\ifx \bpublisher  \undefined \def \bpublisher#1{#1}\fi
\ifx \bbtitle  \undefined \def \bbtitle#1{#1}\fi
\ifx \bedition  \undefined \def \bedition#1{#1}\fi
\ifx \bseriesno  \undefined \def \bseriesno#1{#1}\fi
\ifx \blocation  \undefined \def \blocation#1{#1}\fi
\ifx \bsertitle  \undefined \def \bsertitle#1{#1}\fi
\ifx \bsnm \undefined \def \bsnm#1{#1}\fi
\ifx \bsuffix \undefined \def \bsuffix#1{#1}\fi
\ifx \bparticle \undefined \def \bparticle#1{#1}\fi
\ifx \barticle \undefined \def \barticle#1{#1}\fi
\bibcommenthead
\ifx \bconfdate \undefined \def \bconfdate #1{#1}\fi
\ifx \botherref \undefined \def \botherref #1{#1}\fi
\ifx \url \undefined \def \url#1{\textsf{#1}}\fi
\ifx \bchapter \undefined \def \bchapter#1{#1}\fi
\ifx \bbook \undefined \def \bbook#1{#1}\fi
\ifx \bcomment \undefined \def \bcomment#1{#1}\fi
\ifx \oauthor \undefined \def \oauthor#1{#1}\fi
\ifx \citeauthoryear \undefined \def \citeauthoryear#1{#1}\fi
\ifx \endbibitem  \undefined \def \endbibitem {}\fi
\ifx \bconflocation  \undefined \def \bconflocation#1{#1}\fi
\ifx \arxivurl  \undefined \def \arxivurl#1{\textsf{#1}}\fi
\csname PreBibitemsHook\endcsname

\bibitem[\protect\citeauthoryear{Akaike}{1959}]{Akaike1959}
\begin{barticle}
\bauthor{\bsnm{Akaike}, \binits{H.}}:
\batitle{On a successive transformation of probability distribution and its
  application to the analysis of the optimum gradient method}.
\bjtitle{Annals of the Institute of Statistical Mathematics}
\bvolume{11},
\bfpage{1}--\blpage{16}
(\byear{1959})
\end{barticle}
\endbibitem

\bibitem[\protect\citeauthoryear{Barzilai and Borwein}{1988}]{Barzilai1988}
\begin{barticle}
\bauthor{\bsnm{Barzilai}, \binits{J.}},
\bauthor{\bsnm{Borwein}, \binits{J.M.}}:
\batitle{Two-point step size gradient methods}.
\bjtitle{IMA Journal of Numerical Analysis}
\bvolume{8},
\bfpage{141}--\blpage{148}
(\byear{1988})
\end{barticle}
\endbibitem

\bibitem[\protect\citeauthoryear{Dai and Liao}{2002}]{Dai2002}
\begin{botherref}
\oauthor{\bsnm{Dai}, \binits{Y.-H.}},
\oauthor{\bsnm{Liao}, \binits{L.-Z.}}:
R-linear convergence of the barzilai and borwein gradient method.
IMA Journal of Numerical Analysis
\textbf{22}
(2002)
\end{botherref}
\endbibitem

\bibitem[\protect\citeauthoryear{Dai et~al.}{2006}]{Dai2006}
\begin{barticle}
\bauthor{\bsnm{Dai}, \binits{Y.-H.}},
\bauthor{\bsnm{Hager}, \binits{W.W.}},
\bauthor{\bsnm{Schittkowski}, \binits{K.}},
\bauthor{\bsnm{Zhang}, \binits{H.}}:
\batitle{The cyclic barzilai-–borwein method for unconstrained optimization}.
\bjtitle{IMA Journal of Numerical Analysis}
\bvolume{26}(\bissue{3}),
\bfpage{604}--\blpage{627}
(\byear{2006})
\end{barticle}
\endbibitem

\bibitem[\protect\citeauthoryear{Raydan}{1997}]{Raydan1997}
\begin{botherref}
\oauthor{\bsnm{Raydan}, \binits{M.}}:
The barzilai and borwein gradient method for the large scale unconstrained
  minimization problem.
SIAM Journal on Optimization
\textbf{7}
(1997)
\end{botherref}
\endbibitem

\bibitem[\protect\citeauthoryear{Birgin et~al.}{2000}]{Birgin2000}
\begin{barticle}
\bauthor{\bsnm{Birgin}, \binits{E.G.}},
\bauthor{\bsnm{Mart\'{\i}nez}, \binits{J.M.}},
\bauthor{\bsnm{Raydan}, \binits{M.}}:
\batitle{Nonmonotone spectral projected gradient methods on convex sets}.
\bjtitle{SIAM Journal on Optimization}
\bvolume{10}(\bissue{4}),
\bfpage{1196}--\blpage{1211}
(\byear{2000})
\end{barticle}
\endbibitem

\bibitem[\protect\citeauthoryear{Dai and Fletcher}{2005}]{Dai2005b}
\begin{barticle}
\bauthor{\bsnm{Dai}, \binits{Y.-H.}},
\bauthor{\bsnm{Fletcher}, \binits{R.}}:
\batitle{Projected barzilai-borwein methods for large-scale box-constrained
  quadratic programming}.
\bjtitle{Numerische Mathematik}
\bvolume{100},
\bfpage{21}--\blpage{47}
(\byear{2005})
\end{barticle}
\endbibitem

\bibitem[\protect\citeauthoryear{Dai and Fletcher}{2006}]{Dai2006a}
\begin{barticle}
\bauthor{\bsnm{Dai}, \binits{Y.-H.}},
\bauthor{\bsnm{Fletcher}, \binits{R.}}:
\batitle{New algorithms for singly linearly constrained quadratic programs
  subject to lower and upper bounds}.
\bjtitle{Math. Program.}
\bvolume{106},
\bfpage{403}--\blpage{421}
(\byear{2006})
\end{barticle}
\endbibitem

\bibitem[\protect\citeauthoryear{Hager and Zhang}{2006}]{Hager2006}
\begin{barticle}
\bauthor{\bsnm{Hager}, \binits{W.}},
\bauthor{\bsnm{Zhang}, \binits{H.}}:
\batitle{A new active set algorithm for box constrained optimization}.
\bjtitle{SIAM Journal on Optimization}
\bvolume{17},
\bfpage{526}--\blpage{557}
(\byear{2006})
\end{barticle}
\endbibitem

\bibitem[\protect\citeauthoryear{Serafini et~al.}{2005}]{Serafini2005}
\begin{botherref}
\oauthor{\bsnm{Serafini}, \binits{T.}},
\oauthor{\bsnm{Zanghirati}, \binits{G.}},
\oauthor{\bsnm{Zanni}, \binits{L.}}:
Gradient projection methods for quadratic programs and applications in training
  support vector machines.
Optimization Methods and Software
\textbf{20}
(2005)
\end{botherref}
\endbibitem

\bibitem[\protect\citeauthoryear{Fletcher}{2005}]{Fletcher2005}
\begin{bchapter}
\bauthor{\bsnm{Fletcher}, \binits{R.}}:
\bctitle{On the barzilai-borwein method}.
In: \beditor{\bsnm{Qi}, \binits{L.}},
\beditor{\bsnm{Teo}, \binits{K.}},
\beditor{\bsnm{Yang}, \binits{X.}} (eds.)
\bbtitle{Optimization and Control with Applications},
pp. \bfpage{235}--\blpage{256}.
\bpublisher{Springer},
\blocation{Boston, MA}
(\byear{2005})
\end{bchapter}
\endbibitem

\bibitem[\protect\citeauthoryear{Raydan and Svaiter}{2002}]{Raydan2002}
\begin{barticle}
\bauthor{\bsnm{Raydan}, \binits{M.}},
\bauthor{\bsnm{Svaiter}, \binits{B.}}:
\batitle{Relaxed steepest descent and cauchy-barzilai-borwein method}.
\bjtitle{Computational Optimization and Applications}
\bvolume{21},
\bfpage{155}--\blpage{167}
(\byear{2002})
\end{barticle}
\endbibitem

\bibitem[\protect\citeauthoryear{Dai}{2003}]{Dai2003}
\begin{barticle}
\bauthor{\bsnm{Dai}, \binits{Y.-H.}}:
\batitle{Alternate step gradient method}.
\bjtitle{Optimization}
\bvolume{52}(\bissue{4-5}),
\bfpage{395}--\blpage{415}
(\byear{2003})
\end{barticle}
\endbibitem

\bibitem[\protect\citeauthoryear{Dai and xiang Yuan}{2005}]{Dai2005}
\begin{barticle}
\bauthor{\bsnm{Dai}, \binits{Y.}},
\bauthor{\bsnm{Yuan}, \binits{Y.-x.}}:
\batitle{Analysis of monotone gradient methods}.
\bjtitle{Journal of Industrial and Management Optimization}
\bvolume{1},
\bfpage{181}--\blpage{192}
(\byear{2005})
\end{barticle}
\endbibitem

\bibitem[\protect\citeauthoryear{Frassoldati et~al.}{2008}]{Frassoldati2008}
\begin{barticle}
\bauthor{\bsnm{Frassoldati}, \binits{G.}},
\bauthor{\bsnm{Zanni}, \binits{L.}},
\bauthor{\bsnm{Zanghirati}, \binits{G.}}:
\batitle{New adaptive stepsize selections in gradient methods}.
\bjtitle{Journal of Industrial and Management Optimization}
\bvolume{4},
\bfpage{299}--\blpage{312}
(\byear{2008})
\end{barticle}
\endbibitem

\bibitem[\protect\citeauthoryear{Asmundis et~al.}{2014}]{Asmundis2014}
\begin{barticle}
\bauthor{\bsnm{Asmundis}, \binits{R.D.}},
\bauthor{\bsnm{Serafino}, \binits{D.}},
\bauthor{\bsnm{Hager}, \binits{W.W.}},
\bauthor{\bsnm{Toraldo}, \binits{G.}},
\bauthor{\bsnm{Zhang}, \binits{H.}}:
\batitle{An efficient gradient method using the yuan steplength}.
\bjtitle{Computational Optimization and Applications}
\bvolume{59},
\bfpage{541}--\blpage{563}
(\byear{2014})
\end{barticle}
\endbibitem

\bibitem[\protect\citeauthoryear{Zou and Magoul{\`e}s}{2018}]{Zou2018}
\begin{botherref}
\oauthor{\bsnm{Zou}, \binits{Q.}},
\oauthor{\bsnm{Magoul{\`e}s}, \binits{F.}}:
A new cyclic gradient method adapted to large-scale linear systems.
2018 17th International Symposium on Distributed Computing and Applications for
  Business Engineering and Science (DCABES),
196--199
(2018)
\end{botherref}
\endbibitem

\bibitem[\protect\citeauthoryear{Oviedo~León}{2021}]{OviedoLeon2021}
\begin{botherref}
\oauthor{\bsnm{Oviedo~León}, \binits{H.F.}}:
A cyclic delayed weighted steplength for the gradient method.
Ricerche di Matematica
\textbf{70}
(2021)
\end{botherref}
\endbibitem

\bibitem[\protect\citeauthoryear{{Huang} et~al.}{2022}]{Huang2022}
\begin{botherref}
\oauthor{\bsnm{{Huang}}, \binits{Y.}},
\oauthor{\bsnm{{Dai}}, \binits{Y.-H.}},
\oauthor{\bsnm{{Liu}}, \binits{X.-W.}},
\oauthor{\bsnm{{Zhang}}, \binits{H.}}:
{On the asymptotic convergence and acceleration of gradient methods}.
Journal of Scientific Computing
\textbf{90}
(2022)
\end{botherref}
\endbibitem

\bibitem[\protect\citeauthoryear{Friedlander et~al.}{1998}]{Friedlander1998}
\begin{barticle}
\bauthor{\bsnm{Friedlander}, \binits{A.}},
\bauthor{\bsnm{Mart\'{\i}nez}, \binits{J.M.}},
\bauthor{\bsnm{Molina}, \binits{B.}},
\bauthor{\bsnm{Raydan}, \binits{M.}}:
\batitle{Gradient method with retards and generalizations}.
\bjtitle{SIAM Journal on Numerical Analysis}
\bvolume{36}(\bissue{1}),
\bfpage{275}--\blpage{289}
(\byear{1998})
\end{barticle}
\endbibitem

\bibitem[\protect\citeauthoryear{Dai and Fletcher}{2005}]{Dai2005a}
\begin{barticle}
\bauthor{\bsnm{Dai}, \binits{Y.-H.}},
\bauthor{\bsnm{Fletcher}, \binits{R.}}:
\batitle{On the asymptotic behaviour of some new gradient methods}.
\bjtitle{Mathematical Programming}
\bvolume{103},
\bfpage{541}--\blpage{559}
(\byear{2005})
\end{barticle}
\endbibitem

\bibitem[\protect\citeauthoryear{Yuan}{2006}]{Yuan2006}
\begin{barticle}
\bauthor{\bsnm{Yuan}, \binits{Y.-x.}}:
\batitle{A new stepsize for the steepest descent method}.
\bjtitle{Journal of Computational Mathematics}
\bvolume{24},
\bfpage{149}--\blpage{156}
(\byear{2006})
\end{barticle}
\endbibitem

\bibitem[\protect\citeauthoryear{Yuan}{2008}]{Yuan2008}
\begin{botherref}
\oauthor{\bsnm{Yuan}, \binits{Y.-x.}}:
Step-sizes for the gradient method.
AMS IP Studies in Advanced Mathematics,
785--796
(2008)
\end{botherref}
\endbibitem

\bibitem[\protect\citeauthoryear{Huang et~al.}{2021}]{Huang2021}
\begin{barticle}
\bauthor{\bsnm{Huang}, \binits{Y.-K.}},
\bauthor{\bsnm{Dai}, \binits{Y.-H.}},
\bauthor{\bsnm{Liu}, \binits{X.-W.}}:
\batitle{Equipping the barzilai--borwein method with the two dimensional
  quadratic termination property}.
\bjtitle{SIAM Journal on Optimization}
\bvolume{31}(\bissue{4}),
\bfpage{3068}--\blpage{3096}
(\byear{2021})
\end{barticle}
\endbibitem

\bibitem[\protect\citeauthoryear{Nocedal and Wright}{2006}]{JorgeNocedal2006}
\begin{bbook}
\bauthor{\bsnm{Nocedal}, \binits{J.}},
\bauthor{\bsnm{Wright}, \binits{S.J.}}:
\bbtitle{Numerical Optimization}.
\bpublisher{Springer},
\blocation{United States of America}
(\byear{2006})
\end{bbook}
\endbibitem

\bibitem[\protect\citeauthoryear{Gould et~al.}{2014}]{Gould2014}
\begin{botherref}
\oauthor{\bsnm{Gould}, \binits{N.}},
\oauthor{\bsnm{Orban}, \binits{D.}},
\oauthor{\bsnm{Toint}, \binits{P.}}:
Cutest: a constrained and unconstrained testing environment with safe threads
  for mathematical optimization.
Computational Optimization and Applications
\textbf{60}
(2014)
\doiurl{10.1007/s10589-014-9687-3}
\end{botherref}
\endbibitem

\bibitem[\protect\citeauthoryear{{Huang} et~al.}{2022}]{Huang2022a}
\begin{botherref}
\oauthor{\bsnm{{Huang}}, \binits{Y.}},
\oauthor{\bsnm{{Dai}}, \binits{Y.-H.}},
\oauthor{\bsnm{{Liu}}, \binits{X.-W.}}:
{A mechanism of three-dimensional quadratic termination for the gradient method
  with applications}.
arXiv e-prints,
2212--07255
(2022)
{\href{https://arxiv.org/abs/2212.07255}{{arXiv:2212.07255}}}
{[math.OC]}
\end{botherref}
\endbibitem

\bibitem[\protect\citeauthoryear{{di Serafino} et~al.}{2018}]{diSerafino2018}
\begin{barticle}
\bauthor{\bsnm{{di Serafino}}, \binits{D.}},
\bauthor{\bsnm{Ruggiero}, \binits{V.}},
\bauthor{\bsnm{Toraldo}, \binits{G.}},
\bauthor{\bsnm{Zanni}, \binits{L.}}:
\batitle{On the steplength selection in gradient methods for unconstrained
  optimization}.
\bjtitle{Applied Mathematics and Computation}
\bvolume{318},
\bfpage{176}--\blpage{195}
(\byear{2018})
\doiurl{10.1016/j.amc.2017.07.037} .
\bcomment{Recent Trends in Numerical Computations: Theory and Algorithms}
\end{barticle}
\endbibitem

\bibitem[\protect\citeauthoryear{Dolan and Moré}{2001}]{Dolan2001}
\begin{botherref}
\oauthor{\bsnm{Dolan}, \binits{E.}},
\oauthor{\bsnm{Moré}, \binits{J.}}:
Benchmarking optimization software with performance profiles.
Mathematical Programming
\textbf{91}
(2001)
\doiurl{10.1007/s101070100263}
\end{botherref}
\endbibitem

\end{thebibliography}

\end{document}